\documentclass[11pt,reqno]{amsart}
\usepackage{geometry}                		
\usepackage{graphicx}				
\usepackage{amssymb}
\usepackage{amsmath,amssymb,mathrsfs,amsfonts,amsthm,mathptmx}
\usepackage{graphicx,cite,times,color,bm,dcolumn}
\numberwithin{equation}{section}
\numberwithin{figure}{section}
\numberwithin{table}{section}
\newtheorem{theorem}{Theorem}[section]
\newtheorem{lemma}{Lemma}[section]

\newtheorem{proposition}{Proposition}[section]
\newtheorem{remark}{Remark}[section]

\allowdisplaybreaks[4]
\begin{document}
\title[Operator splitting method]{A new efficient operator splitting method for stochastic Maxwell equations}\thanks{The research of C. Chen and J. Hong were supported by the NNSFC (NOs.  11971470, 11871068, 12022118, and 12031020), the research of L. Ji was supported by the NNSFC (NOs. 11601032, and 11971458)}
\author{Chuchu Chen}
\address{LSEC, ICMSEC, Academy of Mathematics and Systems Science, Chinese Academy of Sciences, and School of Mathematical Sciences, University of Chinese Academy of Sciences, Beijing 100049, China}
\email{chenchuchu@lsec.cc.ac.cn}
\author{Jialin Hong}
\address{LSEC, ICMSEC, Academy of Mathematics and Systems Science, Chinese Academy of Sciences, and School of Mathematical Sciences, University of Chinese Academy of Sciences, Beijing 100049, China}
\email{hjl@lsec.cc.ac.cn}
\author{Lihai Ji}
\address{Institute of Applied Physics and Computational Mathematics, Beijing 100094, People's Republic of China.}
\email{jilihai@lsec.cc.ac.cn}
	\begin{abstract}
	This paper proposes and analyzes a new operator splitting method for stochastic Maxwell equations driven by additive noise, which not only decomposes the original multi-dimensional system into some local one-dimensional subsystems, but also separates the deterministic and stochastic parts. This method is numerically efficient, and preserves the symplecticity, the multi-symplecticity as well as the growth rate of the averaged energy. A detailed $H^2$-regularity analysis of stochastic Maxwell equations is obtained, which is a crucial prerequisite of the error analysis. Under the regularity assumptions of the initial data and the noise, the convergence order one in mean square sense of the operator splitting method is established.
	\end{abstract}

\maketitle

\section{Introduction}\label{sec:1}

This paper  proposes and analyzes  a new operator splitting method for solving the following three-dimensional stochastic Maxwell equations with additive noise:
\begin{subequations}\label{sto_max}
	\begin{align}
&{\rm d}{\bf E}(t)={\rm curl} \;{\bf H}(t){\rm d}t+\lambda_1{\rm d}W(t),\qquad (t,{\bf x})\in (0,T]\times D,\label{sto_max_1}\\[2mm]
&{\rm d}{\bf H}(t)=-{\rm curl} \; {\bf E}(t){\rm d}t+\lambda_2{\rm d}W(t),\qquad (t,{\bf x})\in (0,T]\times D,\label{sto_max_2}\\[2mm]
&{\bf E}(0,{\bf x})={\bf E}_0({\bf x}),~~{\bf H}(0,{\bf x})={\bf H}_0({\bf x}),\qquad {\bf x}\in D,
	\end{align}
\end{subequations}
where $D=(a_1^-,a_1^+)\times (a_2^-,a_2^+)\times (a_3^-,a_3^+)\subseteq {\mathbb R}^{3}$ is  a cuboid, ${\bf E}$ and ${\bf H}$ represent the electric field and the magnetic field, respectively, and $\lambda_i\in{\mathbb R}^3$, $i=1,2,$ describe the amplitude of the noise. Here, $\{W(t)\}_{t\in[0,T]}$ is a standard $Q$-Wiener process with respect to a filtered probability space $(\Omega,{\mathcal F},\{{\mathcal F}_{t}\}_{0\leq t\leq T},{\mathbb P})$ with $Q$ being a  symmetric, positive definite operator with finite trace on a Hilbert space $U$.

Stochastic Maxwell equations provide the foundation in  stochastic electromagnetism and statistical radiophysics etc., when considering thermodynamic fluctuations or wave propagation in random field (cf. \cite{RKT1989}).  It is of special importance to develop efficient numerical methods for 
 simulating stochastic Maxwell equations in large scale and long time computations.
In the last decade, various numerical methods for stochastic Maxwell equations have been developed, analyzed, and tested in the literature. For instance, see \cite{C2020,Zhang2008} for the finite element method and discontinuous Galerkin method, \cite{BAZC2010} for  Wiener chaos expansion method, \cite{CHJ2019a,CHJ2019b, CCHS2020} for temporally semi-discrete methods, \cite{CHZ2016,HJZ2014,HJZC2017} for stochastic multi-symplectic methods.

Taking the multi-dimensional character of stochastic Maxwell equations into consideration,  the operator splitting method (or the dimension splitting method) is one of the most promising methods in the reduction of computational costs.
The splitting technique  is introduced to the context of stochastic partial differential equations (SPDEs) in \cite{BG1989} for Zakai equation in filtering, and is extended to the general linear SPDEs in \cite{GK2003}, in order to decompose the vector field into the stochastic and deterministic parts. Afterwards, there are many papers using the splitting technique, and we refer to \cite{CH2017,CHLZ2019,L2013b} for the case of stochastic Schr\"odinger equation,  to \cite{BCH2019,BG2019} for the case of stochastic Allen--Cahn equation, and to \cite{BM2019,D2012} for the case of stochastic Navier--Stokes equations, etc. 

In this paper,
we propose a new operator splitting method for stochastic Maxwell equations, which combines the superiority of the decomposability of the original multi-dimensional system into some local one-dimensional subsystems (see c.f. \cite{CLL2010, HJS2015,KHZ2010} for the  case of deterministic Maxwell equations) and the feature of the separation of the deterministic and stochastic parts (see c.f. \cite{BG1989, GK2003}). 
To be specific, stochastic Maxwell equations \eqref{sto_max} are split into the following subsystems
\begin{equation}\label{sect1: sub1}
\left\{
\begin{aligned}
{\rm d}E_1^{[1]}(t)=\lambda_1^1{\rm d} W(t)\\
{\rm d}H_1^{[1]}(t)=\lambda_2^1{\rm d}W(t)
\end{aligned}
\right. ,\quad
\left\{
\begin{aligned}
{\rm d}E_2^{[1]}(t)=-\partial_x H_3^{[1]} (t){\rm d}t\\
{\rm d}H_3^{[1]}(t)=-\partial_x E_2^{[1]}(t){\rm d}t
\end{aligned}
\right. ,\quad
\left\{
\begin{aligned}
{\rm d} E_3^{[1]}(t)=\partial_x H_2^{[1]}(t) {\rm d}t\\
{\rm d}H_2^{[1]}(t)=\partial_x E_3^{[1]}(t){\rm d}t
\end{aligned}
\right. ,
\end{equation}
\begin{equation}\label{sect1: sub2}
\left\{
\begin{aligned}
{\rm d}E_2^{[2]}(t)=\lambda_1^2{\rm d}W(t)\\
{\rm d}H_2^{[2]}(t)=\lambda_2^2{\rm d}W(t)
\end{aligned}
\right. ,\quad
\left\{
\begin{aligned}
{\rm d}E_1^{[2]}(t)=\partial_y H_3^{[2]}(t){\rm d}t\\
{\rm d} H_3^{[2]}(t)=\partial_y E_1^{[2]}(t) {\rm d}t
\end{aligned}
\right. ,\quad
\left\{
\begin{aligned}
{\rm d} E_3^{[2]}(t)=-\partial_y H_1^{[2]}(t) {\rm d}t\\
{\rm d}H_1^{[2]}(t)=-\partial_y E_3^{[2]}(t){\rm d}t
\end{aligned}
\right. ,
\end{equation}
\begin{equation}\label{sect1: sub3}
\left\{
\begin{aligned}
{\rm d} E_3^{[3]}(t)=\lambda_1^3{\rm d}W(t)\\
{\rm d} H_3^{[3]}(t)=\lambda_2^3{\rm d}W(t)
\end{aligned}
\right. ,\quad
\left\{
\begin{aligned}
{\rm d}E_1^{[3]}(t)=-\partial_z H_2^{[3]}(t){\rm d}t\\
{\rm d}H_2^{[3]}(t)=-\partial_z E_1^{[3]}(t){\rm d}t
\end{aligned}
\right. ,\quad
\left\{
\begin{aligned}
{\rm d}E_2^{[3]}(t)=\partial_z H_1^{[3]}(t){\rm d}t\\
{\rm d}H_1^{[3]}(t)=\partial_z E_2^{[3]}(t){\rm d}t
\end{aligned}
\right. .
\end{equation}
Thus
an approximation of \eqref{sto_max} in the interval $[t_{n-1},t_{n}]$ is split into three steps: solving the subsystem \eqref{sect1: sub1} and taking its solution at time $t_{n}$ as the initial value at time $t_{n-1}$ while solving the subsystem \eqref{sect1: sub2} again on $[t_{n-1},t_{n}]$, then repeating it to the subsystem \eqref{sect1: sub3}. Namely, the approximation $\{(\tilde{\bf E}(t_{n}),\tilde{\bf H}(t_{n}))\}_{1\leq n \leq N}$ is defined recurrently by 
\begin{equation}\label{sec1:eq1}
(\tilde{\bf E}(t_{n}),\tilde{\bf H}(t_{n}))=\Psi^{[3]}_{t_{n-1},t_{n}}\circ \Psi^{[2]}_{t_{n-1},t_{n}} \circ \Psi^{[1]}_{t_{n-1},t_{n}}(\tilde{\bf E}(t_{n-1}),\tilde{\bf H}(t_{n-1})),\quad n\geq 1,
\end{equation}
 where $(\tilde{\bf E}(t_{0}),\tilde{\bf H}(t_{0}))=({\bf E}_0,{\bf H}_0)$, and $\{\Psi^{[j]}_{s,t}:0\leq s\leq t\}$, $j=1,2,3$ denote  solution flows of  subsystems \eqref{sect1: sub1}-\eqref{sect1: sub3}, respectively.

The proposed decompositions of Maxwell operator and the vector $\lambda_i$ ($i=1,2$) are such that the components $E_1^{[1]}$ and $H_1^{[1]}$ in subsystem \eqref{sect1: sub1} only contain the stochastic part whose exact solutions can be given explicitly, and other components $E_2^{[1]}, E_3^{[1]},H_2^{[1]}, H_3^{[1]}$ only contain the deterministic part (i.e., two one-dimensional wave equations) whose spatial derivative  is in one space direction. The same holds for \eqref{sect1: sub2} and \eqref{sect1: sub3}. This feature significantly reduces the complexity of the problem and improves the computational efficiency. 
Moreover, we show that each subsystem is still a stochastic Hamiltonian PDE preserving both the infinite-dimensional stochastic symplectic structure and stochastic multi-symplectic conservation law. If each subsystem is solved exactly, then the operator splitting method \eqref{sec1:eq1} preserves the averaged energy exactly.

In order to study the mean square convergence of the operator splitting method \eqref{sec1:eq1}, a relevant prerequisite  is to provide the regularity theory of the solution of \eqref{sto_max}. Utilizing the analysis for mixed inhomogeneous boundary value problems for the Laplacian on a cuboid (\cite[Lemma 3.1]{ES2017}), the $L^2(\Omega; H^2(D)^6)$-regularity ($H^2$-regularity for short) for the solution of \eqref{sto_max} is proved  with certain assumptions being made only on the initial data and the noise; see Proposition \ref{pro_2.2}. Our main result states that for each $T>0$, there exists a constant $C$ independent of $\tau$ and $N$ such that 
\begin{equation}
\max_{0\leq n\leq N}\Big[{\mathbb E}\Big(\|{\bf E}(t_n)-\tilde{\bf E}(t_n)\|_{L^2(D)^3}^2+\|{\bf H}(t_n)-\tilde{\bf H}(t_n)\|_{L^2(D)^3}^2\Big)\Big]^{\frac12}\leq C\tau,
\end{equation}
where $\tau$ is the uniform step size.
Furthermore, it follows from a straightforward modification of the proof for this error estimate that the result still holds if we perturb the order of the subsystems, for example,  the splitting method by using another  order $\Psi^{[1]}_{t_{n-1},t_{n}}\circ \Psi^{[2]}_{t_{n-1},t_{n}} \circ \Psi^{[3]}_{t_{n-1},t_{n}}$ and so on. 

The paper is organized as follows. In Section 2, we introduce the basic setting and give the regularity analysis of  stochastic Maxwell equations. Section 3 presents the formulation of the operator splitting method and analyzes the preservation of stochastic symplecticity and multi-symplecticity by the phase flow of each subsystem. In Section 4, we establish error estimates in mean square sense of the operator splitting method.

\section{Properties of stochastic Maxwell equations}\label{preliminaries}
This section presents the regularity analysis of stochastic Maxwell equations. The regularity in $L^2(\Omega;H^k(D)^6)$ with $k=1,2$ is proved in detail, based on the analysis for mixed inhomogeneous boundary value problems for the Laplacian on a cuboid. Throughout this paper,  denote by $C$ a generic positive constant which may be different from line to line, but independent of the step size $\tau$ and the partition number $N$.
\subsection{Notation}
In this paper, we adopt the following conventions.  We write $Id$ for the identity operator and $v\cdot w$ for the Euclidean inner product. ${\bf E}=(E_1,E_2,E_3)^{\top}$ and ${\bf H}=(H_1,H_2,H_3)^{\top}$ represent the electric field and the magnetic field, respectively, and $\lambda_i=(\lambda_i^{1},\lambda_i^{2},\lambda_i^{3})^{\top}\in{\mathbb R}^3,$ $i=1,2,$ describe the amplitude of the noise.
	
We work on a finite time interval $[0,T]$, $T>0$, and on the cuboid $D=(a_1^-,a_1^+)\times (a_2^-,a_2^+)\times (a_3^-,a_3^+)\subseteq {\mathbb R}^{3}$ with Lipschitz boundary $\Gamma:=\partial D$. 
Denote
$
\Gamma_j^{\pm}:=\{x\in\bar{D}|~x_j=a_j^{\pm}\}$
and $\Gamma_j:=\Gamma_j^-\cup\Gamma_j^+
$
for $j\in\{1,2,3\}$. Then $\Gamma=\Gamma_1\cup \Gamma_2\cup \Gamma_3$.
Let ${\bf n}$ be the unit outward normal to $\Gamma$. 
Let $H^s(D)$, $s\in{\mathbb R}$ be the classical Sobolev space. The basic Hilbert space we work with is $V:=L^2(D)^6$ with inner product $\langle\cdot,\cdot\rangle_V$ and norm $\|\cdot\|_V$. Let us define two spaces related to the curl operator
\begin{align*}
&H({\rm curl},D)=\left\{ v\in L^2(D)^3:~~ {\rm curl}\; v\in L^2(D)^3  \right\},\\[1mm]
&H_0({\rm curl},D)=\left\{ v\in H({\rm curl},D):~~ {\bf n}\times v|_{\Gamma}=0  \right\}
\end{align*}
endowed with the norm
\begin{align*}
\|v\|_{H({\rm curl},D)}=\left( \|v\|^2_{L^2(D)^3}+\|{\rm curl}\; v\|^2_{L^2(D)^3} \right)^{\frac{1}{2}}.
\end{align*}

We enforce a perfectly electric conducting (PEC) boundary condition on $D$,
\begin{equation}\label{bc:E}
{\bf n}\times{\bf E}=0,\quad \text{on}\quad  [0,T]\times\Gamma.
\end{equation}

An important skew-adjoint operator which will be often used throughout the paper is the Maxwell operator
\begin{equation}\label{maxwell_operator}
M:=\left(\begin{array}{cc}
0 & {\rm curl} \\
-{\rm curl} & 0
\end{array}\right)
\end{equation}
with domain
\begin{equation}\label{domain}
\begin{split}
{\mathcal D}(M)&=\left\{\begin{pmatrix}
{\bf E} \\
{\bf H}
\end{pmatrix}\in V:~M\begin{pmatrix}
{\bf E} \\
{\bf H}
\end{pmatrix}=\begin{pmatrix}
 {\rm curl}\;{\bf H}\\
-{\rm curl}\;{\bf E}
\end{pmatrix}\in V,~ {\bf n}\times{\bf E}\Big|_{\Gamma}=0 \right\}\\[2mm]
&=H_0({\rm curl},D)\times H({\rm curl},D).
\end{split}
\end{equation}
The corresponding graph norm is $\|v\|_{\mathcal{D}(M)}:=(\|v\|_V^2+\|Mv\|_V^2)^{1/2}$. Recursively we could define the domain ${\mathcal D}(M^k):=\{u\in {\mathcal D}(M^{k-1}):~M^{k-1}u\in {\mathcal D}(M)\}$ of the  $k$-th power of the operator $M$, $k\in{\mathbb N}$, with ${\mathcal D}(M^0)=V$.

Let $U:=\{f\in L^2(D): ~f=0\mbox{ on }\Gamma\}$. If we denote an orthonormal basis of  $U$ by $\{e_k\}_{k\in{\mathbb N}}$, the Karhunen-Lo\`{e}ve expansion
yields
$
W(t,{\bf x})=\sum_{k\in{\mathbb N}}Q^{\frac12}e_k({\bf x})\beta_k(t),~t\in[0,\,T],~{\bf x}\in D,
$
with $\{\beta_k\}_{k\in{\mathbb N}}$ being a family of independent real-valued Brownian motions.
\begin{remark}
We point out that, under the assumption ${\bf n}\cdot {\bf H}_0|_{\Gamma}=0$, one can verify that the PEC boundary condition \eqref{bc:E} is equivalent to 
\begin{equation}\label{bc:H}
{\bf n}\cdot {\bf H}=0,\quad \text{on}\quad  [0,T]\times\Gamma.
\end{equation}
In fact, since $Q^{\frac12}h|_{\Gamma}=0$, for any $h\in U$,  \eqref{bc:E} together with \eqref{sto_max_2} leads to
\begin{equation*}
\begin{split}
{\rm d}({\bf n}\cdot{\bf H}(t))&=-{\bf n}\cdot {\rm curl}\;{\bf E}(t){\rm d}t+{\bf n}\cdot \lambda_2{\rm d}W(t)\\[0.5mm]
&=\left[\nabla\cdot({\bf n}\times{\bf E})-{\bf E}\cdot {\rm curl}\;{\bf n}\right]{\rm d}t+\sum_{k=1}^{\infty} {\bf n}\cdot\lambda_2 Q^{\frac12}e_{k}{\rm d}\beta_{k}(t)=0,
\end{split}
\end{equation*}
due to the identity $\nabla \cdot (u\times v)=v\cdot {\rm curl}~ u-u\cdot {\rm curl}\; v$ for the second equation.
The last equation holds because of \eqref{bc:E} and the fact that ${\bf n}$ can be written as the gradient of a parametrization of $\Gamma$ and then ${\rm curl}(\nabla f)=0$ for some scalar function $f$.
Thus one has
$
{\bf n}\cdot {\bf H}={\bf n}\cdot {\bf H}_0=0$ on  $ [0,T]\times\Gamma.
$
\end{remark}
\subsection{Well-posedness} 
Let us give a short comment here on the well-posedness of stochastic Maxwell equations and  the uniform boundedness of the solution  in $L^2(\Omega;{\mathcal D}(M^k))$-norm. 
We note that, with the operator $M$ defined in \eqref{maxwell_operator}, stochastic Maxwell equations \eqref{sto_max} can be rewritten as
\begin{equation}\label{abstract form}
{\rm d}{\bf Z}(t)=M{\bf Z}(t){\rm d}t+\lambda{\rm d}W(t),~~t\in(0,T]; \quad{\bf Z}(0)={\bf Z}_0,
\end{equation}
where ${\bf Z}=({\bf E}^{\top},{\bf H}^{\top})^{\top}$, ${\bf Z}_0=({\bf E}_0^{\top},{\bf H}_0^{\top})^{\top}$ and $\lambda=(\lambda_1^{\top},\lambda_2^{\top})^{\top}$. 
\begin{lemma}{\rm \cite{EN2000}}
	The Maxwell operator $M$ defined in \eqref{maxwell_operator} with domain ${\mathcal D}(M)$ is closed, skew-adjoint on $V$, and generates a unitary $C_0$-semigroup $S(t)=e^{tM}$ on $V$ for $t\in\mathbb{R}^+$.
\end{lemma}

Similar to the well-posedness and uniform boundedness results from \cite[Propositions 2.1 and 2.2]{CHJ2019b}, we have the following theorem.
\begin{theorem}
Let ${\bf Z}_0$ be an $\mathcal{F}_0$-measurable $V$-valued random variable satisfying $\|{\bf Z}_0\|_{L^2(\Omega;V)}<\infty$, and let $Q$ be a symmetric, positive definite, trace class operator on $U$. Then stochastic Maxwell equations \eqref{abstract form} have a unique mild solution ${\bf Z}\in L^2(\Omega;C([0,T]; V))$  satisfying 
\begin{equation}\label{mild sol}
{\bf Z}(t)=S(t){\bf Z}_0+\int_0^t S(t-s)\lambda{\rm d}W(s),\quad \mathbb{P}\text{-}a.s.,
\end{equation}
for each $t\in[0,T]$.

Furthermore, if for any $k\in \mathbb{N}$, $Q^{\frac{1}{2}}\in HS(U,H^{k}(D))$ and $\|{\bf Z}_0\|_{L^2(\Omega,\mathcal{D}(M^k))}<\infty$, then the mild solution ${\bf Z}\in L^2(\Omega;C([0,T]; \mathcal{D}(M^k)))$ and satisfies 
\begin{equation}\label{eq 2.8}
{\mathbb E}\Big[\sup_{t\in[0,T]}\|{\bf Z}(t)\|^2_{{\mathcal D}(M^k)}\Big]\leq C\Big(1+{\mathbb E}\|{\bf Z}_0\|^2_{{\mathcal D}(M^k)}\Big),
\end{equation}
where the constant $C$ may depend on $T$, $|\lambda|$ and $\|Q^{\frac12}\|_{HS( U, H^{k}(D))}$.  Here $HS(U,H^k(D))$ denotes the family of Hilbert-Schmidt operators from $U$ to $H^k(D)$.
\end{theorem}

The averaged energy  of stochastic Maxwell equations \eqref{sto_max} evolutes linearly with the rate $|\lambda|^2{\rm Tr}(Q)$ (see \cite[Theorem 2.1]{CHZ2016}):
\begin{equation}\label{ave-energy}
{\mathbb E}\|{\bf Z}(t)\|^2_{V}={\mathbb E}\|{\bf Z}_0\|^2_{V}+t|\lambda|^2{\rm Tr}(Q).
\end{equation}

From the proof of  \cite[Theorem 2.2]{CHZ2016}, one gets the divergence evolution laws for the stochastic Maxwell equations. Let $H({\rm div},D)=\{ v\in (L^2(D))^3:~~\nabla\cdot v\in L^2(D)^3  \}.$
\begin{lemma} 
Assume that ${\bf E}_0,{\bf H}_0\in H({\rm div},D)$, $Q^{\frac{1}{2}}\in HS(U,H^{1}(D))$.
The divergence of system \eqref{sto_max} satisfies
\begin{align}\label{div}
\nabla\cdot {\bf E}(t)=\nabla\cdot {\bf E}_0+\lambda_1\cdot \left(\nabla W(t)\right),\quad \nabla\cdot {\bf H}(t)=\nabla\cdot {\bf H}_0+\lambda_2\cdot \left(\nabla W(t)\right).
\end{align}
\end{lemma}

\subsection{$H^k$-regularity results}
In this part, we prove the regularities of the mild solution \eqref{mild sol} in $L^2(\Omega; H^k(D)^6)$-norm ($H^k$-regularity, for short) with $k=1,2$, which are essential in the mean square convergence analysis.

\begin{proposition}\label{H1 regularity}
Assume that $Q^{\frac12}\in {HS( U, H_0^1(D))}$ and ${\bf Z}_0\in L^2(\Omega; H^1(D)^6)$. Then the solution  \eqref{mild sol} has the $H^1$-regularity, i.e.,
\[
{\mathbb E}\|{\bf Z}(t)\|_{H^1(D)^6}^2\leq C\left(1+{\mathbb E}\|{\bf Z}_0\|_{H^1(D)^6}^2\right),
\]
where  the constant $C$ may depend on $T$,  $|\lambda|$  and $\|Q^{\frac12}\|_{HS( U, H^1(D))}$. 
\end{proposition}
\begin{proof}
The $H^1$-regularity is deduced by utilizing the fact that  $v\in H({\rm curl},D)\cap H({\rm div},D)$ belongs to $H^1(D)^3$ if $v\times {\bf n}=0$ or $v\cdot{\bf n}=0$  on $\Gamma$. Moreover, the $H^1$-norm of $v$ is dominated by
 \begin{equation}\label{eq 2.6}
\|v\|_{H^1(D)^3}\leq C\left( \|v\|_{L^2(D)^3}+\|{\rm curl}\; v\|_{L^2(D)^3}
+\|\nabla\cdot v\|_{L^2(D)}\right).
\end{equation}

Meanwhile, from  \eqref{div}, it can be verified that there exists a constant
 $C$ depending on $T,|\lambda_1|,\|Q^{\frac12}\|_{HS( U, H^1(D))}$ such that
\begin{equation}\label{eq 2.10}
{\mathbb E}\|\nabla\cdot {\bf E}(t)\|_{L^2(D)}^2\leq 2 {\mathbb E}\|\nabla\cdot {\bf E}_0\|_{L^2(D)}^2+2{\mathbb E}\|\lambda_1\cdot \left(\nabla W(t)\right)\|_{L^2(D)}^2\leq C(1+{\mathbb E}\|\nabla\cdot {\bf E}_0\|_{L^2(D)}^2).
\end{equation}
Similarly, it holds
\begin{equation}\label{100}
{\mathbb E}\|\nabla\cdot {\bf H}(t)\|_{L^2(D)}^2\leq C(1+{\mathbb E}\|\nabla\cdot {\bf H}_0\|_{L^2(D)}^2)
\end{equation}
with $C$ depending on $T$, $|\lambda_2|$ and $\|Q^{\frac12}\|_{HS( U, H^1(D))}$.

Since ${\bf n}\times{\bf E}=0$ and ${\bf n}\cdot {\bf H}=0$ on $[0,T]\times\Gamma$, from \eqref{eq 2.6}, we obtain
\begin{equation*}
\begin{split}
{\mathbb E}\|{\bf Z}(t)\|^2_{H^1(D)^6}\leq &C\bigg(
{\mathbb E}\|{\bf E}(t)\|^2_{L^2(D)^3}+{\mathbb E}\|{\bf H}(t)\|^2_{L^2(D)^3}
+{\mathbb E}\|{\rm curl}\;{\bf E}(t)\|^2_{L^2(D)^3}\\
&+{\mathbb E}\|{\rm curl}\;{\bf H}(t)\|^2_{L^2(D)^3}+{\mathbb E}\|\nabla\cdot{\bf E}(t)\|^2_{L^2(D)}+{\mathbb E}\|\nabla\cdot{\bf H}(t)\|^2_{L^2(D)}
\bigg)\\
\leq &C\bigg({\mathbb E}\|{\bf Z}(t)\|^2_{{\mathcal D}(M)} +{\mathbb E}\|\nabla\cdot{\bf E}(t)\|^2_{L^2(D)}+{\mathbb E}\|\nabla\cdot{\bf H}(t)\|^2_{L^2(D)}\bigg)\\
\leq &C(1+{\mathbb E}\|{\bf Z}_0\|_{H^1(D)^6}^2),
\end{split}
\end{equation*}
due to  \eqref{eq 2.8}, \eqref{eq 2.10} and \eqref{100}.
The proof is completed.
\end{proof}

In our error analysis we still need the $H^2$-regularity of the solution \eqref{mild sol}, whose proof relies on the following lemma about the mixed inhomogeneous boundary value problems for the Laplacian on a cuboid; see \cite[Lemma 3.1]{ES2017} for the detailed proof. 
\begin{lemma}{\rm \cite[Lemma 3.1]{ES2017}}\label{mixed}
Let $j\in\{1,2,3\}$ and $\Gamma^*=\Gamma \backslash \Gamma_j$. Take $f\in L^2(D)$ and $g\in H_0^{1/2}(\Gamma_j):=\big(L^2(\Gamma_j),~H_0^1(\Gamma_j)\big)_{1/2,2}$.
If there is a unique function $v\in H_{\Gamma^*}^{1}(D)$ solving
\begin{equation}
\int_D v\varphi {\rm d}{\bf x} +\int_D \nabla v\cdot \nabla\varphi {\rm d}{\bf x}
=\int_{D} f\varphi{\rm d}{\bf x}+\int_{\Gamma_j^{+}}g\varphi {\rm d}\sigma -\int_{\Gamma_j^-}g\varphi {\rm d}\sigma,
\end{equation}
for all $\varphi\in H_{\Gamma^*}^1(D)$, then the solution $v$ belongs to $H^2(D)\cap H_{\Gamma^*}^1(D)$ and satisfies $v-\Delta v=f$ on $D$, $\partial_{\bf n}v=g$ on $\Gamma_j$, and $\|v\|_{H^2(D)}\leq C\big( \|f\|_{L^2(D)}+\|g\|_{H_0^{1/2}(\Gamma_j)} \big)$ with the constant $C$ only depending on $D$.
\end{lemma}

Following the approach in \cite{ES2017,HJS2015},
we introduce a subspace $H_{00}^1(D)$ of $H^1(D)$ as
\[
H_{00}^1(D)=\left\{f\in H^1(D)~|~ {\rm tr}_{\Gamma'}f\in H_0^{1/2}(\Gamma'),~~\mbox{for all faces }\Gamma'\mbox{ of }D \right\},
\]
and achieve higher regularity of the solution \eqref{mild sol}.
\begin{proposition}\label{pro_2.2}
Assume that $Q^{\frac12}\in {HS\left( U, H^2(D)\cap H_0^1(D)\right)}$, $\nabla Q^{\frac12}: ~U\to H_{00}^1(D)^3$, and ${\bf Z}_0\in L^2\left(\Omega; H^2(D)^6\right)$, $\nabla\cdot {\bf E}_0\in L^2(\Omega; H_{00}^1(D))$. Then the solution \eqref{mild sol} has  the $H^2$-regularity, i.e.,
\[
{\mathbb E}\|{\bf Z}(t)\|_{H^2(D)^6}^2\leq C(1+{\mathbb E}\|{\bf Z}_0\|_{H^2(D)^6}^2).
\]
For any $t\in[0,T]$, the field $({\bf E}(t),{\bf H}(t))$ has the traces
\begin{equation*}
\begin{split}
E_j=E_k=0&,\quad \partial_j E_j=\partial_k E_j=\partial_j E_k=\partial_k E_k=0 ,\quad \mbox{on } \Gamma_i,\\[1mm]
H_i=0&,~~\qquad\qquad\qquad\quad  \partial_jH_i=\partial_k H_i=0,\quad \mbox{on } \Gamma_i,
\end{split}
\end{equation*}
for all permutations $(i,j,k)$ of $(1,2,3)$.
\end{proposition}
\begin{proof}
{\it Step 1.} The $H^1$-regularity follows from Proposition \ref{H1 regularity}. Moreover, the asserted zero-order traces for ${\bf E}$ and ${\bf H}$ are a direct consequence of the boundary conditions \eqref{bc:E} and \eqref{bc:H}, respectively. The first-order traces result from the established zero-order traces and the following $H^2$-regularity.

From \eqref{div}, one gets
\[
{\mathbb E}\|\nabla\cdot{\bf E}(t)\|_{H^1(D)}^2+{\mathbb E}\|\nabla\cdot{\bf H}(t)\|_{H^1(D)}^2
\leq C(1+\|Z_0\|^2_{H^2(D)^6}),
\]
where the constant $C$ may depend on $T$, $|\lambda|$ and $\|Q^{\frac12}\|_{HS( U, H^2(D))}$. Moreover from assumptions of the initial data and noise, we have that for any $t\in[0,T]$, $\nabla\cdot{\bf E}(t), ~\nabla\cdot{\bf H}(t)\in L^2(\Omega; H_{00}^1(D))$.

We note that
$$
M^2{\bf Z}=\begin{pmatrix}
-{\rm curl}\left({\rm curl}\;{\bf E}\right)\\[1mm]
-{\rm curl}\left({\rm curl}\;{\bf H}\right)
\end{pmatrix},
$$
and $\Delta {\bf E}=\nabla(\nabla\cdot{\bf E})-{\rm curl}\left({\rm curl}\;{\bf E}\right)$. It follows from \eqref{eq 2.8} with $k=2$ that $\Delta {\bf E}\in L^2(\Omega; L^2(D)^3)$, and
\[
{\mathbb E}\|\Delta {\bf E}(t)\|^2_{L^2(D)^3}\leq C({\mathbb E}\|\nabla\cdot{\bf E}(t)\|_{H^1(D)}^2+{\mathbb E}\|{\bf Z}(t)\|^2_{{\mathcal D}(M^2)}).
\]
The field ${\bf H}$ can be estimated similarly. Standard interior elliptic regularity then leads to ${\bf E}(t), ~{\bf H}(t)\in L^2(\Omega; H^2_{\rm loc}(D)^3)$.

{\it Step 2.} We first consider the $H^2$-regularity of the first component $E_1$ of ${\bf E}$. Set $\Gamma^*=\Gamma_2\cup\Gamma_3$. From {\em Step 1} it implies that $f:=(I-\Delta)E_1\in L^2(D)$. By employing cut-off and mollification in $y,z$ directions,  one can approximate a given $\varphi\in H_{\Gamma^*}^1(D)$  in $H^1(D)$  by a smooth $\psi$ having support in 
$[a_1^-,a_1^+]\times [a_2^-+\eta, a_2^+-\eta]\times[a_3^-+\eta,a_3^+-\eta]$ for some small $\eta:=\eta(\psi)>0$.
For each $\kappa\in(0,\eta)$, define $D_{\kappa}:=(a_1^-+\kappa,a_1^+-\kappa)\times (a_2^-+\kappa, a_2^+-\kappa)\times(a_3^-+\kappa,a_3^+-\kappa)$  and denote by $\Gamma_1^{\pm}(\kappa)$ those open faces of $D_{\kappa}$ that contain the points of the form $(a_1^{\mp}\pm\kappa,y,z)$. Integration by parts and the support of $\psi$ yield that
\begin{equation}
\begin{split}
\int_D  & E_1\psi{\rm d}{\bf x} +\int_{D} \nabla E_1\cdot \nabla\psi{\rm d}{\bf x}
=\lim_{\kappa\to 0}\int_{D_{\kappa}} E_1\psi{\rm d}{\bf x} +\int_{D_{\kappa}} \nabla E_1\cdot \nabla\psi{\rm d}{\bf x}\\
&=\lim_{\kappa\to 0}\left[\int_{D_{\kappa}} \psi (I-\Delta)E_1{\rm d}{\bf x}
+ \int_{\partial D_{\kappa}} \psi \nabla E_1\cdot{\bf n}{\rm d}\sigma \right]\\
&=\int_{D} \psi f{\rm d}{\bf x}
\pm \lim_{\kappa\to 0}\int_{\Gamma_1^{\pm}(\kappa)} \psi \partial_x E_1 {\rm d}\sigma \\
&=\int_{D} \psi f{\rm d}{\bf x}
\pm \lim_{\kappa\to 0}\int_{\Gamma_1^{\pm}(\kappa)} \psi \big( \nabla\cdot {\bf E}-\partial_y E_2-\partial_z E_3\big) {\rm d}\sigma \\
&=\int_{D} \psi f{\rm d}{\bf x} +\int_{\Gamma_1^+} \psi \nabla\cdot {\bf E} {\rm d}\sigma -\int_{\Gamma_1^-}\psi \nabla\cdot {\bf E} {\rm d}\sigma,
\end{split}
\end{equation}
where the last step is due to integration by parts once more and the fact that $\psi$ vanishes on the boundary of $\Gamma_1^{\pm}(\kappa)$, as well as $E_2,~ E_3$ on $\Gamma_1$.
Lemma \ref{mixed} and $\nabla\cdot{\bf E}\in L^2(\Omega; H_{00}^1(D))$  (see {\em Step 1}) lead to  $E_1\in L^2(\Omega;H^2(D))$. In the same manner, one observes that $E_2,E_3\in L^2(\Omega;H^2(D))$. Furthermore
\begin{equation*}
\begin{split}
{\mathbb E}\|E_j\|_{H^2(D)}^2 &\leq C\left({\mathbb E}\|E_j\|^2_{L^2(D)}+{\mathbb E}\|\Delta E_j\|^2_{L^2(D)} +{\mathbb E}\|\nabla\cdot{\bf E}\|^2_{H_0^{1/2}(\Gamma_j)} \right)\\
&\leq C\left({\mathbb E}\|\nabla\cdot{\bf E}\|_{H^1(D)}^2+{\mathbb E}\|{\bf Z}\|^2_{{\mathcal D}(M^2)}+{\mathbb E}\|\nabla\cdot{\bf E}\|^2_{H_0^{1/2}(\Gamma_j)} \right).
\end{split}
\end{equation*}

{\it Step 3.} Now we consider $H_1$, and set $\Gamma^*=\Gamma_1$ and $\widetilde{f}:=(I-\Delta)H_1\in L^2(D)$. Here we use the homogeneous version of Lemma \ref{mixed} (see also \cite[Lemma 3.6]{HJS2015}). As in {\it Step 2}, we take a smooth function $\psi\in H^1(D)$ having support in $[a_1^-+\eta,a_1^+-\eta]\times [a_2^-, a_2^+]\times[a_3^-,a_3^+]$ for some small $\eta:=\eta(\psi)>0$. Choose $\kappa\in(0,\eta)$ so that $\psi$ vanishes around $\Gamma_1^{\pm}(\kappa)$, then
\begin{equation}
\begin{split}
\int_D  & H_1\psi{\rm d}{\bf x} +\int_{D} \nabla H_1\cdot \nabla\psi{\rm d}{\bf x}
=\lim_{\kappa\to 0}\left[\int_{D_{\kappa}} H_1\psi{\rm d}{\bf x} +\int_{D_{\kappa}} \nabla H_1\cdot \nabla\psi{\rm d}{\bf x}\right]\\
&=\lim_{\kappa\to 0}\left[\int_{D_{\kappa}} \psi (I-\Delta)H_1{\rm d}{\bf x}
+ \int_{\partial D_{\kappa}} \psi \nabla H_1\cdot{\bf n}{\rm d}\sigma \right]\\
&=\int_{D} \psi \widetilde{f}{\rm d}{\bf x}
+ \lim_{\kappa\to 0}\int_{\partial D_{\kappa}} \Big[\psi \nabla H_1\cdot{\bf n}-\big(({\rm curl}\;{\bf H})\times {\bf n}\big)\cdot(\psi,0,0)\Big]{\rm d}\sigma \\
&=\int_{D} \psi \widetilde{f}{\rm d}{\bf x}
+ \lim_{\kappa\to 0}\int_{\partial D_{\kappa}} \psi \partial_x {\bf H}\cdot {\bf n}{\rm d}\sigma\\
&=\int_{D} \psi \widetilde{f}{\rm d}{\bf x}
\pm \lim_{\kappa\to 0}\left[\int_{\Gamma_2^{\pm}(\kappa)} \psi \partial_x H_2{\rm d}\sigma+\int_{\Gamma_3^{\pm}(\kappa)} \psi \partial_x H_3{\rm d}\sigma\right]\\
&=\int_{D} \psi \widetilde{f}{\rm d}{\bf x},
\end{split}
\end{equation}
where in the last step, we use again integration by parts and the fact that $H_2,~ H_3$ vanishes on $\Gamma_2$, $\Gamma_3$, respectively.
Hence
\[
{\mathbb E}\|H_1\|_{H^2(D)}^2\leq C\big( {\mathbb E}\|H_1\|_{L^2(D)}^2+{\mathbb E}\|\Delta H_1\|_{L^2(D)}^2 \big)\leq C\big( {\mathbb E}\|\nabla\cdot {\bf H}\|^2_{H^1(D)}+ {\mathbb E}\| {\bf Z}\|^2_{{\mathcal D}(M^2)}\big).
\]
Components $H_2,~H_3$ can be treated similarly.

Combining {\it Steps 1-3} and \eqref{eq 2.8}, we have
\begin{equation*}
\begin{split}
{\mathbb E}\|{\bf Z}(t)\|^2_{H^2(D)^6}&\leq C\Big({\mathbb E}\|{\bf Z}(t)\|^2_{{\mathcal D}(M^2)}+{\mathbb E}\|\nabla\cdot{\bf E}(t)\|_{H^1(D)}^2+{\mathbb E}\|\nabla\cdot {\bf H}(t)\|^2_{H^1(D)}\\
&\quad+\sum_{\Gamma'\in\Gamma}{\mathbb E}\|\nabla\cdot{\bf E}(t)\|^2_{H_0^{1/2}(\Gamma')} \Big)\\
&\leq C(1+{\mathbb E}\|{\bf Z}_0\|_{H^2(D)^6}^2).
\end{split}
\end{equation*}
Thus the proof is finished.
\end{proof}

\section{Operator splitting method}\label{section3}
In this section, we introduce the dimension splitting of the Maxwell operator. By employing a corresponding decomposition of $\lambda_i$ $(i=1,2)$, we get three subsystems, which enjoy the superiority of the decomposability of the original multi-dimensional problem into some local one-dimensional subproblems and the feature of the separation of the deterministic and stochastic parts. Moreover, each subsystem is a stochastic Hamiltonian system. 
 
\subsection{Operator splitting}\label{subsection_3.1}
Let us now describe how to split three-dimensional stochastic Maxwell equations \eqref{sto_max} into some one-dimensional subsystems. We split the `${\rm curl}$' operator into
\begin{equation}
  {\rm curl}={\rm curl}_x+{\rm curl}_y+{\rm curl}_z,
\end{equation}
where 
\begin{equation*}
{\rm curl}_x=\begin{bmatrix}
0 &0 & 0\\
0 & 0 & -\partial_x\\
0 & \partial_x & 0
\end{bmatrix},\quad{\rm curl}_y=\begin{bmatrix}
0 &0 &  \partial_y\\
0 & 0 & 0\\
-\partial_y & 0 & 0
\end{bmatrix},\quad{\rm curl}_z=\begin{bmatrix}
0 & -\partial_z &0\\
\partial_z & 0 &0\\
0 &0 & 0
\end{bmatrix},
\end{equation*}
are one-dimensional diffenrential operators (see e.g. \cite{CLL2010,KHZ2010} for the case of  deterministic Maxwell equations).
Define operators
\begin{align}\label{operators}
M_{x}=\begin{bmatrix}0 & {\rm curl}_x \\ -{\rm curl}_x & 0 \end{bmatrix},
\quad
M_{y}=\begin{bmatrix}0 & {\rm curl}_y \\ -{\rm curl}_y & 0 \end{bmatrix},
\quad
M_{z}=\begin{bmatrix}0 & {\rm curl}_z \\ -{\rm curl}_z & 0 \end{bmatrix},
\end{align}
on $V$ endowed with domains
\begin{align*}
&{\mathcal D}(M_x)=\big\{u\in V:~~M_x u\in V,~~ u_2=u_3=0~~\mbox{on}~~\Gamma_1 \big\},\\
&{\mathcal D}(M_y)=\big\{u\in V:~~M_y u\in V,~~ u_1=u_3=0~~\mbox{on}~~\Gamma_2 \big\},\\
&{\mathcal D}(M_z)=\big\{u\in V:~~M_z u\in V,~~ u_1=u_2=0~~\mbox{on}~~\Gamma_3 \big\},
\end{align*}
respectively.
Note that $$Mu=M_x u+M_y u+M_z u$$ for $u\in {\mathcal D}(M_x)\cap {\mathcal D}(M_y)\cap {\mathcal D}(M_z)\subset {\mathcal D}(M)$. Then, stochastic Maxwell equations \eqref{sto_max} can be decomposed into the following three subsystems
\begin{align}
    &{\rm d}{\bf Z}^{[1]}(t)=M_x {\bf Z}^{[1]}(t){\rm d}t+\lambda^{[1]}{\rm d}W(t), \label{system_1}\\[1.5mm]
    &{\rm d}{\bf Z}^{[2]}(t)=M_y {\bf Z}^{[2]}(t){\rm d}t+\lambda^{[2]}{\rm d}W(t), \label{system_2}\\[1.5mm]
    &{\rm d}{\bf Z}^{[3]}(t)=M_z {\bf Z}^{[3]}(t){\rm d}t+\lambda^{[3]}{\rm d}W(t),\label{system_3}
\end{align}
where $\lambda^{[1]}=(\lambda_1^1,0,0,\lambda_2^1,0,0)^{\top}$,
$\lambda^{[2]}=(0,\lambda_1^2,0,0,\lambda_2^2,0)^{\top}$, $\lambda^{[3]}=(0,0,\lambda_1^3,0,0,\lambda_2^3)^{\top}$.

\subsection{Properties of subsystems}
The splitting of the original system \eqref{sto_max} into the three subsystems \eqref{system_1}-\eqref{system_3} gives an approach to simulating stochastic Maxwell equations effectively and efficiently  in large scale and long time computations.

To show this clearly, we rewrite \eqref{system_1}-\eqref{system_3} into the following componentwise forms, respectively,
\begin{equation}\label{eq 4.55}
\left\{
\begin{aligned}
{\rm d}E_1^{[1]}(t)=\lambda_1^1{\rm d} W(t)\\
{\rm d}H_1^{[1]}(t)=\lambda_2^1{\rm d}W(t)
\end{aligned}
\right. ,\quad
\left\{
\begin{aligned}
{\rm d}E_2^{[1]}(t)=-\partial_x H_3^{[1]} (t){\rm d}t\\
{\rm d}H_3^{[1]}(t)=-\partial_x E_2^{[1]}(t){\rm d}t
\end{aligned}
\right. ,\quad
\left\{
\begin{aligned}
{\rm d} E_3^{[1]}(t)=\partial_x H_2^{[1]}(t) {\rm d}t\\
{\rm d}H_2^{[1]}(t)=\partial_x E_3^{[1]}(t){\rm d}t
\end{aligned}
\right. ,
\end{equation}
\begin{equation}\label{eq 4.56}
\left\{
\begin{aligned}
{\rm d}E_2^{[2]}(t)=\lambda_1^2{\rm d}W(t)\\
{\rm d}H_2^{[2]}(t)=\lambda_2^2{\rm d}W(t)
\end{aligned}
\right. ,\quad
\left\{
\begin{aligned}
{\rm d}E_1^{[2]}(t)=\partial_y H_3^{[2]}(t){\rm d}t\\
{\rm d} H_3^{[2]}(t)=\partial_y E_1^{[2]}(t) {\rm d}t
\end{aligned}
\right. ,\quad
\left\{
\begin{aligned}
{\rm d} E_3^{[2]}(t)=-\partial_y H_1^{[2]}(t) {\rm d}t\\
{\rm d}H_1^{[2]}(t)=-\partial_y E_3^{[2]}(t){\rm d}t
\end{aligned}
\right. ,
\end{equation}
\begin{equation}\label{eq 4.57}
\left\{
\begin{aligned}
{\rm d} E_3^{[3]}(t)=\lambda_1^3{\rm d}W(t)\\
{\rm d} H_3^{[3]}(t)=\lambda_2^3{\rm d}W(t)
\end{aligned}
\right. ,\quad
\left\{
\begin{aligned}
{\rm d}E_1^{[3]}(t)=-\partial_z H_2^{[3]}(t){\rm d}t\\
{\rm d}H_2^{[3]}(t)=-\partial_z E_1^{[3]}(t){\rm d}t
\end{aligned}
\right. ,\quad
\left\{
\begin{aligned}
{\rm d}E_2^{[3]}(t)=\partial_z H_1^{[3]}(t){\rm d}t\\
{\rm d}H_1^{[3]}(t)=\partial_z E_2^{[3]}(t){\rm d}t
\end{aligned}
\right. .
\end{equation}
Note that \eqref{eq 4.55}-\eqref{eq 4.57} possess similar characters and structures. Taking \eqref{eq 4.55} for an example, 
 for components $E_1^{[1]}$ and $H_1^{[1]}$, they only contain the stochastic part and admit the explicit formulas of the exact solutions. And for other components $E_2^{[1]}, H_3^{[1]}$ and $E_3^{[1]},H_2^{[1]}$, they only contain the deterministic part with spatial derivative being in one space direction. They are two one-dimensional wave equations, which can be solved independently.
These characters of the splitting will lead to a dramatic reduction of computational costs in solving stochastic Maxwell equations.

\subsubsection{Well-posedness}
The skew-adjointness of  operators $M_x$, $M_y$ and $M_z$ is shown in the following lemma.
\begin{lemma}
Operators $M_\alpha~(\alpha=x,y,z)$ with domain ${\mathcal D}(M_{\alpha})$ are skew-adjoint on $V$.
\end{lemma}

\begin{proof}
To prove the skew-adjointness of $M_\alpha$ it is enough to show that $M_\alpha$ is a skew-symmetric operator and that $Id\pm M_\alpha$ has dense range. To show the skew-symmetry of $M_x$, we take $\psi=(u^{\top},~v^{\top})^{\top}$ and $\tilde{\psi}=(\tilde{u}^{\top},~\tilde{v}^{\top})^{\top}$ in ${\mathcal D}(M_x)$. The integration by parts formula then implies
\begin{align*}
\left\langle M_x\psi, \tilde{\psi}\right\rangle_V
&=\int_{D} \left(-\tilde{u}_2\partial_x v_3+\tilde{u}_3\partial_x v_2 +\tilde{v}_2\partial_x u_3-\tilde{v}_3\partial_x u_2 \right){\rm d}x{\rm d}y{\rm d}z\\[1.5mm]
&=\int_{D} \left(v_3 \partial_x\tilde{u}_2-v_2\partial_x\tilde{u}_3-u_3\partial_x\tilde{v}_2+u_2 \partial_x\tilde{v}_3\right){\rm d}x{\rm d}y{\rm d}z\\[1.5mm]
&=-\left\langle \psi,  M_x\tilde{\psi}\right\rangle_V,
\end{align*}
due to the boundary conditions in the definition of ${\mathcal D}(M_x)$. Thus $M_x$ is skew-symmetric and analogously for $M_y,\,M_z$.

To check the density of $Id\pm M_x$, we have to show that
\begin{equation}\label{1.42}
  \overline{{\rm ran}(Id\pm M_x)}=V.
\end{equation}
Because $C^{\infty}(D)^6$ is dense in $V$, we infer that \eqref{1.42} is equivalent to show that for every
$f\in C^{\infty}(D)^6$ there is a $g=({\bf E}^{\top},{\bf H}^{\top})^{\top}\in {\mathcal D}(M_x)$ such that
\begin{equation}\label{eq3.10}
(Id\pm M_x)g=f,
\end{equation}
or equivalently,
\begin{equation*}
  \begin{split}
    &E_1=f_1,\quad E_2 \mp \partial_x H_3 =f_2,\quad E_3\pm \partial_x H_2=f_3,\\[1.5mm]
    &H_1=f_4,\quad H_2\mp \partial_x E_3=f_5,\quad H_3\pm \partial_x E_2=f_6.
  \end{split}
\end{equation*}
It yields
\begin{align*}
E_2-\partial_{xx} E_2=f_2\pm \partial_x f_6=:\widetilde{f}_2\in L^2(D),\\[1.5mm]
E_3 -\partial_{xx} E_3=f_3\mp \partial_x f_5=:\widetilde{f}_3\in L^2(D).
\end{align*}
In order to solve these equations, we introduce the operator $A_x=\partial_{xx}$ with domain
\[
{\mathcal D}(A_x)=\{u\in L^2(D):~~\partial_x u,~A_x u\in L^2(D),~~u=0 \mbox{ on } \Gamma_1 \}.
\]
The Lax-Milgram lemma thus yields the existence of $E_2$, $E_3\in {\mathcal D}(A_x)$ such that $E_j-A_x E_j=\widetilde{f}_j,\, j=2,3$. Defining
\[
H_2=f_5\pm\partial_x E_3,\qquad H_3=f_6\mp\partial_x E_2,
\]
we obtain a solution $g=({\bf E}^{\top},{\bf H}^{\top})^{\top}\in {\mathcal D}(M_x)$ of \eqref{eq3.10}. Similarly, we can get the results for $M_y$ and $M_z$. Thus the proof is finished.
\end{proof}

By applying Stone's theorem, operators $M_x$, $M_y$ and $M_z$ generate unitary $C_0$-semigroups $S_x(t)=e^{tM_x}$, $S_y(t)=e^{tM_y}$ and $S_z(t)=e^{tM_z}$ on $V$ for $t\in \mathbb{R}^+$, respectively. Therefore each subsystem has a unique mild solutions  given by
\begin{align}
{\bf Z}^{[1]}(t)=S_{x}(t){\bf Z}_0^{[1]}+\lambda^{[1]}\int_0^{t} S_{x}(t-s){\rm d}W(s),\\
{\bf Z}^{[2]}(t)=S_{y}(t){\bf Z}_0^{[2]} +\lambda^{[2]}\int_0^{t}S_{y}(t-s){\rm d}W(s),\\
{\bf Z}^{[3]}(t)=S_{z}(t){\bf Z}_0^{[3]} +\lambda^{[3]}\int_0^{t}S_{z}(t-s){\rm d}W(s),
\end{align}
$\mathbb{P}$-a.s., respectively.

\begin{proposition}\label{sub-energy}
Let ${\bf Z}_0^{[j]}$ be ${\mathcal F}_0$-measurable $V$-valued random variables satisfying $\|{\bf Z}_0^{[j]}\|_{L^2(\Omega;V)}<\infty$ for $j=1,2,3$, and let $Q$ be a symmetric, positive definite, trace class operator on $U$. Then
\begin{align*}
{\mathbb E}\|{\bf Z}^{[j]}(t)\|^2_{V}={\mathbb E}\|{\bf Z}^{[j]}_0\|^2_{V}+t|\lambda^{[j]}|^2{\rm Tr}(Q).
\end{align*}
\end{proposition}
\subsubsection{Stochastic symplecticity and multi-symplecticity}
The phase flow of stochastic Maxwell equations \eqref{sto_max} conserves the infinite-dimensional stochastic symplectic structure (\cite{CHJ2019b}) and stochastic multi-symplectic conservation law (\cite{CHZ2016, HJZ2014}). This part investigates these geometric structures for subsystems \eqref{system_1}-\eqref{system_3}.

To present the formulation of stochastic Hamiltonian system for each subsystem, we define 
\begin{equation*}
\begin{split}
\mathcal{H}_1^{[j]}=\frac{1}{2}\int_D\Big(|{\bf E}^{[j]}|^2+|{\bf H}^{[j]}|^2\Big){\rm d}{\bf x},\quad
\mathcal{H}_2^{[j]}=\int_D\Big(\lambda_1^{j}H_j^{[j]}-\lambda_2^{j}E_j^{[j]}\Big){\rm d}{\bf x},\\[1mm]
\end{split}
\end{equation*}
where $j=1,2,3$. Then subsystems \eqref{system_1}-\eqref{system_3} are reformulated as 
\begin{equation}\label{3.11}
\begin{bmatrix}
{\rm d}{\bf E}^{[j]}(t)\\[3mm]
{\rm d}{\bf H}^{[j]}(t)
\end{bmatrix}=\begin{bmatrix}
0 & {\rm curl}_{\alpha}\\[3mm]
-{\rm curl}_{\alpha} & 0
\end{bmatrix}
\begin{bmatrix}
\frac{\delta\mathcal{H}_1^{[j]}}{\delta {\bf E}^{[j]}}\\[3mm]
\frac{\delta\mathcal{H}_1^{[j]}}{\delta {\bf H}^{[j]}}
\end{bmatrix}
{\rm d}t
+\begin{bmatrix}
0 & Id\\[2mm]
-Id & 0
\end{bmatrix}
\begin{bmatrix}
\frac{\delta\mathcal{H}_2^{[j]}}{\delta {\bf E}^{[j]}}\\[3mm]
\frac{\delta\mathcal{H}_2^{[j]}}{\delta {\bf H}^{[j]}}
\end{bmatrix}
\circ{\rm d}W(t),
\end{equation}
 for $(\alpha,j)\in\{(x,1),(y,2),(z,3)\}$.
\begin{lemma}\label{symplectic_flow}
The phase flows of subsystems \eqref{system_1}, \eqref{system_2} and \eqref{system_3} preserve symplectic structures
\begin{equation*}
\overline{{\omega}}^{[j]}(t)=\overline{\omega}^{[j]}(0), \quad {\mathbb P}\text{-}a.s., \quad j=1,2,3,
\end{equation*}
respectively, where 
\begin{equation*}
\overline{\omega}^{[j]}(t):=\int_D{ d}{\bf E}^{[j]}(t,{\bf x})\wedge {\rm curl}_{\alpha}{ d}{\bf H}^{[j]}(t,{\bf x}){\rm d}{\bf x},
\end{equation*}
are the integral over the space domain of the differential 2-forms  with ‘d’ denoting the exterior derivative
for $(\alpha,j)\in\{(x,1),(y,2),(z,3)\}$, respectively. 
\end{lemma}
\begin{proof}
Taking the exterior differential on both sides of \eqref{3.11} and utilizing the skew-adjointness of $M_{\alpha}$ on ${\mathcal D}(M_{\alpha})$ yield the conclusion.
\end{proof}

We remark that the canonical formulations of stochastic Hamiltonian system  for subsystems  \eqref{system_1}-\eqref{system_3} are 
\begin{equation}
\begin{bmatrix}
{\rm d}{\bf E}^{[j]}(t)\\[3mm]
{\rm d}{\bf H}^{[j]}(t)
\end{bmatrix}=\begin{bmatrix}
0 & Id\\[3mm]
-Id & 0
\end{bmatrix}
\begin{bmatrix}
\frac{\delta\mathcal{H}_1^{[\alpha]}}{\delta {\bf E}^{[j]}}\\[3mm]
\frac{\delta\mathcal{H}_1^{[\alpha]}}{\delta {\bf H}^{[j]}}
\end{bmatrix}
{\rm d}t
+\begin{bmatrix}
0 & Id\\[2mm]
-Id & 0
\end{bmatrix}
\begin{bmatrix}
\frac{\delta\mathcal{H}_2^{[j]}}{\delta {\bf E}^{[j]}}\\[3mm]
\frac{\delta\mathcal{H}_2^{[j]}}{\delta {\bf H}^{[j]}}
\end{bmatrix}
\circ{\rm d}W(t),
\end{equation}
 with \begin{equation*}
\begin{split}
\mathcal{H}_1^{[\alpha]}&=\frac{1}{2}\int_D\Big({\bf E}^{[j]}\cdot{\rm curl}_{\alpha}{\bf E}^{[j]}+{\bf H}^{[j]}\cdot{\rm curl}_{\alpha}{\bf H}^{[j]}\Big){\rm d}{\bf x},
\end{split}
\end{equation*}
where $(\alpha,j)\in\{(x,1),(y,2),(z,3)\}$ respectively. In this case, the symplectic structures 
\begin{equation*}
\overline{\omega}^{[j]}(t):=\int_Dd{\bf E}^{[j]}(t,{\bf x})\wedge d{\bf H}^{[j]}(t,{\bf x}){\rm d}{\bf x},\quad j=1,2,3
\end{equation*}
are preserved by the phase flows of subsystems \eqref{system_1}-\eqref{system_3}
respectively if  zero boundary conditions are enforced (see \cite[Theorem 3.2]{CHJ2019b}).


Next we consider the stochastic multi-symplecticity of subsystems  \eqref{system_1}-\eqref{system_3}. Let $u^{[j]}=(H_1^{[j]},H_2^{[j]},H_3^{[j]},E_1^{[j]},E_2^{[j]},E_3^{[j]})^{\top}, j=1, 2, 3$, and denote skew-symmetric matrices
\begin{equation*}\label{FK}
F=\begin{bmatrix}
0 & Id\\[3mm]
-Id & 0
\end{bmatrix},\quad
K_{j}=\begin{bmatrix}
\mathcal{D}_{j}&0\\[3mm]
0&\mathcal{D}_{j}\\
\end{bmatrix},
\end{equation*}
with
$$
	\mathcal{D}_{1}=\left(\begin{array}{ccccccc}
	0&0&0\\[1mm]
	0&0&-1\\[1mm]
	0&~1&0\\[1mm]
	\end{array}\right),~~
	\mathcal{D}_{2}=\left(\begin{array}{ccccccc}
	0&0&~1\\[1mm]
	0&0&0\\[1mm]
	-1&0&0\\[1mm]
	\end{array}\right),~~
	\mathcal{D}_{3}=\left(\begin{array}{ccccccc}
	0&-1&0\\[1mm]
	~1&0&0\\[1mm]
	0&0&0\\[1mm]
	\end{array}\right), $$
and Hamiltonian
$$S(u^{[j]})=\lambda_2^{j}H_1^{[j]}-\lambda_1^{j}E_1^{[j]}.$$
Then,  \eqref{system_1}-\eqref{system_3} can be written as
\begin{equation}
F{\rm d}u^{[j]}+K_j \partial_{\alpha}u^{[j]}{\rm d}t=\nabla S(u^{[j]})\circ {\rm d}W,
\end{equation}
with $(\alpha,j)\in\{(x,1),(y,2),(z,3)\}$, respectively. Similar to the proof of \cite[Theorem 2.3]{CHZ2016}, we  get the following stochastic multi-symplecticity.
\begin{lemma}
Subsystems \eqref{system_1}, \eqref{system_2} and \eqref{system_3} possess  stochastic multi-symplectic conservative laws, respectively, i.e., for $(\alpha,j)\in\{(x,1),(y,2),(z,3)\}$,
\begin{equation*}
{\rm d}\omega^{[j]}+\partial_\alpha\kappa^{[j]}{\rm d}t=0,\quad {\mathbb P}\text{-}a.s.,
\end{equation*}
where 
$$\omega^{[j]}(t,{\bf x})=\frac{1}{2}du^{[j]}\wedge Fdu^{[j]},\kappa^{[j]}(t,{\bf x})=\frac{1}{2}du^{[j]}\wedge K_jdu^{[j]},$$
are  differential 2-forms associated with  skew-symmetric matrices $F$ and $K_j$.
\end{lemma}

\section{Error estimates of the operator splitting method}
This section presents the error analysis of the operator splitting method. Moreover, the errors caused by temporally semi-discretizing   subsystems are also analyzed.

Let $t_n=n\tau ~(n=0,1,\cdots,N)$ be a uniform partition of $[0,T]$ with $\tau=T/N$, and let $\{\Psi^{[j]}_{s,t}:0\leq s\leq t\} ~(j=1,2,3)$ be  solution flows of subsystems \eqref{system_1}-\eqref{system_3}, respectively. 
Let $\tilde{{\bf Z}}(t_{n-1})=(\tilde{\bf E}(t_{n-1}),\tilde{\bf H}(t_{n-1}))$ denote the approximated solution at time $t_{n-1}$, and let $\tilde{{\bf Z}}(t_{0})={\bf Z}_0$.
The detailed algorithm for solving stochastic Maxwell equations \eqref{sto_max} on the interval $[t_{n-1},t_{n}]$ with $n=1,\cdots,N$ is given  by
\begin{itemize}
	\item {\bf Stage 1.} Take $\tilde{{\bf Z}}(t_{n-1})$ as the initial data, solve \eqref{system_1} or \eqref{eq 4.55}, and obtain $\Psi^{[1]}_{t_{n-1},t_{n}}\big(\tilde{{\bf Z}}(t_{n-1})\big)$;\\
	\item {\bf Stage 2.} Take $\Psi^{[1]}_{t_{n-1},t_{n}}\tilde{{\bf Z}}(t_{n-1})$ as the initial data, solve \eqref{system_2} or \eqref{eq 4.56}, and obtain $\Psi^{[2]}_{t_{n-1},t_{n}} \circ \Psi^{[1]}_{t_{n-1},t_{n}}\big(\tilde{{\bf Z}}(t_{n-1})\big)$;\\
	\item {\bf Stage 3.} Take $\Psi^{[2]}_{t_{n-1},t_{n}} \circ \Psi^{[1]}_{t_{n-1},t_{n}}\tilde{{\bf Z}}(t_{n-1})$ as the initial data, solve \eqref{system_3} or \eqref{eq 4.57}, and obtain $\tilde{{\bf Z}}(t_{n})=\Psi_{t_{n-1},t_{n}}\big(\tilde{{\bf Z}}(t_{n-1})\big)$
	with $\Psi_{t_{n-1},t_{n}}:=\Psi^{[3]}_{t_{n-1},t_{n}}\circ \Psi^{[2]}_{t_{n-1},t_{n}} \circ \Psi^{[1]}_{t_{n-1},t_{n}}$.
\end{itemize}

Thus, the composed solution at time $t_{n}=n\tau$ reads
\begin{equation}\label{eq 4.1}
\tilde{{\bf Z}}(t_{n})=\Psi_{t_{n-1},t_{n}}\circ \Psi_{t_{n-2},t_{n-1}}\circ\cdots\circ \Psi_{t_1,t_{2}}\circ \Psi_{t_0,t_{1}}
({\bf Z}_0).
\end{equation}
Note that all the results in this paper still hold for other splittings by changing the order of subsystems \eqref{system_1}-\eqref{system_3}, for example, by using another order $\Psi_{t_{n-1},t_{n}}:=\Psi^{[1]}_{t_{n-1},t_{n}}\circ \Psi^{[2]}_{t_{n-1},t_{n}} \circ \Psi^{[3]}_{t_{n-1},t_{n}}$ and so on.

If $\Psi^{[j]}_{s,t}$ denotes  the exact solution flow of each subsystem, then based on Proposition \ref{sub-energy}, it can be seen that the averaged energy is preserved exactly by the splitting method \eqref{eq 4.1}, i.e.,
\[
{\mathbb E}\|\tilde{{\bf Z}}(t_{n})\|_{V}^2={\mathbb E}\|{\bf Z}_0\|_{V}^2+t_n |\lambda|^2{\rm Tr}(Q).
\]

\subsection{Splitting error}\label{spl_err}
In this subsection we denote $\Psi^{[j]}_{s,t}$  the exact solution flow of each subsystem and establish the mean square convergence analysis of the splitting algorithm for stochastic Maxwell equations \eqref{sto_max}. 	
To present the formulation of $\Psi_{t_n,t_{n+1}}$, we note that the mild solution of the subsystem \eqref{system_1} satisfies
\begin{equation*}
\Psi_{t_n,t_{n+1}}^{[1]}(\tilde{{\bf Z}}(t_{n}))=S_{x}(\tau)\tilde{{\bf Z}}(t_{n})+ \int_{t_n}^{t_{n+1}}S_x(t_{n+1}-r)\lambda^{[1]}{\rm d}W(r).
\end{equation*}
Similarly for \eqref{system_2} starting from $\Psi_{t_n,t_{n+1}}^{[1]}(\tilde{{\bf Z}}(t_{n}))$ at time $t_n$, 
\begin{align*}
\Psi_{t_n,t_{n+1}}^{[2]}\circ \Psi_{t_n,t_{n+1}}^{[1]}(\tilde{{\bf Z}}(t_{n}))=&S_{y}(\tau)\Psi_{t_n,t_{n+1}}^{[1]}(\tilde{{\bf Z}}(t_{n}))+ \int_{t_n}^{t_{n+1}}S_y(t_{n+1}-r)\lambda^{[2]}{\rm d}W(r)\\
=&S_{y}(\tau)S_{x}(\tau)\tilde{{\bf Z}}(t_{n})+\int_{t_n}^{t_{n+1}}S_y(\tau)S_x(t_{n+1}-r)\lambda^{[1]}{\rm d}W(r)\\
&+ \int_{t_n}^{t_{n+1}}S_y(t_{n+1}-r)\lambda^{[2]}{\rm d}W(r),
\end{align*}
and then for \eqref{system_3} starting from $\Psi_{t_n,t_{n+1}}^{[2]}\circ \Psi_{t_n,t_{n+1}}^{[1]}(\tilde{{\bf Z}}(t_{n}))$ at time $t_n$, 
\begin{equation}\label{eq_4.3}
\begin{split}
\tilde{{\bf Z}}(t_{n+1})=&\Psi_{t_n,t_{n+1}}(\tilde{{\bf Z}}(t_{n}))=\Psi_{t_n,t_{n+1}}^{[3]}\circ \Psi_{t_n,t_{n+1}}^{[2]}\circ \Psi_{t_n,t_{n+1}}^{[1]}(\tilde{{\bf Z}}(t_{n}))\\[2mm]
=&S_{z}(\tau)\Psi_{t_n,t_{n+1}}^{[2]}\circ \Psi_{t_n,t_{n+1}}^{[1]}(\tilde{{\bf Z}}(t_{n}))+ \int_{t_n}^{t_{n+1}}S_z(t_{n+1}-r)\lambda^{[3]}{\rm d}W(r)\\[2mm]
=&S_{z}(\tau)S_{y}(\tau)S_{x}(\tau)\tilde{{\bf Z}}(t_{n})+\int_{t_n}^{t_{n+1}}S_{z}(\tau)S_y(\tau)S_x(t_{n+1}-r)\lambda^{[1]}{\rm d}W(r)\\[2mm]
& + \int_{t_n}^{t_{n+1}}S_z(\tau)S_y(t_{n+1}-r)\lambda^{[2]}{\rm d}W(r) + \int_{t_n}^{t_{n+1}}S_z(t_{n+1}-r)\lambda^{[3]}{\rm d}W(r).
\end{split}
\end{equation}

Denote the splitting error  at time $t_n$ by $e^n:={\bf Z}(t_n)-\tilde{{\bf Z}}(t_n)$, $n=0,1,\cdots,N$. We now give the mean square convergence analysis of the error.
\begin{theorem}\label{thm_4.1}
Under the same assumption as in Proposition \ref{pro_2.2}, for sufficiently small $\tau$, there exists a constant $C$ independent of $\tau$ such that
	\begin{equation}\label{err_1}
	\max_{0\leq n\leq N}\mathbb{E}\left[\|e_n\|_{V}^2\right]\leq C\tau^2.
	\end{equation}
\end{theorem}
\begin{proof}

Since $\lambda=\lambda^{[1]}+\lambda^{[2]}+\lambda^{[3]}$, the mild solution of \eqref{abstract form} is
\begin{equation}\label{eq_4.4}
\begin{split}
{\bf Z}(t_{n+1})=S(\tau)&{\bf Z}(t_n)+\int_{t_n}^{t_{n+1}}S(t_{n+1}-r)\lambda^{[1]}{\rm d}W(r)\\
&+\int_{t_n}^{t_{n+1}}S(t_{n+1}-r)\lambda^{[2]}{\rm d}W(r)+\int_{t_n}^{t_{n+1}}S(t_{n+1}-r)\lambda^{[3]}{\rm d}W(r).
\end{split}
\end{equation}
Subtracting \eqref{eq_4.3} from \eqref{eq_4.4}, we obtain
\begin{align}\label{eq_4.5}
e_{n+1}=&S_z(\tau)S_y(\tau)S_x(\tau)e_n+ (S(\tau)-S_z(\tau)S_y(\tau)S_x(\tau)){\bf Z}(t_n)\nonumber\\
&+\int_{t_n}^{t_{n+1}}[S(t_{n+1}-r)-S_{z}(\tau)S_y(\tau)S_x(t_{n+1}-r)]\lambda^{[1]}{\rm d}W(r)\nonumber\\
&+\int_{t_n}^{t_{n+1}}[S(t_{n+1}-r)-S_z(\tau)S_y(t_{n+1}-r)]\lambda^{[2]}{\rm d}W(r)\\
&+\int_{t_n}^{t_{n+1}}[S(t_{n+1}-r)-S_z(t_{n+1}-r)]\lambda^{[3]}{\rm d}W(r)\nonumber\\
&=:S_z(\tau)S_y(\tau)S_x(\tau)e_n+I_1+I_2+I_3+I_4.\nonumber
\end{align}
Next, we give the estimates of terms $I_j$, $j=1,\ldots, 4$. 
\vskip 0.5em
(i) {\it  Estimate of term $I_1$}
\vskip 0.5em
Following \cite{HO2008},
for the generator $F$ of a $C_0$-semigroup and a real number $\tau\geq 0$, we define the bounded operators $\alpha_0(\tau F)=e^{\tau F}$ and
\[
\alpha_{k}(\tau F)=\int_0^1 e^{(1-\xi)\tau F}\frac{\xi^{k-1}}{(k-1)!}{\rm d}\xi,\quad \mbox{for }k\geq 1.
\]
These operators satisfy the recurrence relation
\[
\alpha_{k}(\tau F)=\frac{1}{k!}Id+\tau F \alpha_{k+1}(\tau F), \qquad k\geq 0.
\]

Notice that $S(\tau){\bf Z}(t_n)$ is the solution of the problem $\frac{\rm d}{{\rm d}t}{\bf Z}(t)=(M_z+M_{y}+M_{x}){\bf Z}(t)$ at time $t=t_n$. It can be rewritten by the variation-of-constants formula
\small\begin{equation}\label{eq 4.2}
\begin{split}
&{\bf Z}(t_{n+1})=S(\tau){\bf Z}(t_n)\\
&=S_{z}(\tau){\bf Z}(t_n)+\int_{t_{n}}^{t_{n+1}} S_{z}(s)M_{y}S(t_{n+1}-s){\bf Z}(t_n){\rm d}s+\int_{t_{n}}^{t_{n+1}} S_{z}(s)M_{x}S(t_{n+1}-s){\bf Z}(t_n){\rm d}s\\
&=S_{z}(\tau){\bf Z}(t_n)+\tau S_{z}(\tau)M_y {\bf Z}(t_n)-\int_{t_{n}}^{t_{n+1}}sS_{z}(s)\big( M_z M_y-M_y M \big) S(t_{n+1}-s){\bf Z}(t_n){\rm d}s\\
&\quad+\tau S_{z}(\tau)M_x {\bf Z}(t_n)-\int_{t_{n}}^{t_{n+1}}sS_{z}(s)\big( M_z M_x-M_x M \big) S(t_{n+1}-s){\bf Z}(t_n){\rm d}s,
\end{split}
\end{equation}
\normalsize
where the  integration  by parts formula is utilized in the last step.
For the term $S_{z}(\tau){\bf Z}(t_n)$, we use the relations
\[
S_y(\tau)=\alpha_0(\tau M_y)=Id+\tau M_y +\tau^2 \alpha_2(\tau M_y)M_y^2
\]
and 
\[
S_x(\tau)=\alpha_0(\tau M_x)=Id+\tau M_x +\tau^2 \alpha_2(\tau M_x)M_x^2
\]
to get
\begin{align*}
S_{z}(\tau){\bf Z}(t_n)&=S_{z}(\tau)\Big[S_y(\tau)-\tau M_y -\tau^2 \alpha_2(\tau M_y)M_y^2\Big] {\bf Z}(t_n)\\[2mm]
&=S_{z}(\tau)S_y(\tau){\bf Z}(t_n)- \tau S_{z}(\tau)M_y{\bf Z}(t_n)-\tau^2 S_{z}(\tau)\alpha_2(\tau M_y)M_y^2{\bf Z}(t_n)\\[2mm]
&=S_{z}(\tau)S_y(\tau)\Big[S_x(\tau)-\tau M_x -\tau^2 \alpha_2(\tau M_x)M_x^2 \Big]{\bf Z}(t_n)\\[2mm]
&\quad- \tau S_{z}(\tau)M_y{\bf Z}(t_n)-\tau^2 S_{z}(\tau)\alpha_2(\tau M_y)M_y^2{\bf Z}(t_n)\\[2mm]
&=S_{z}(\tau)S_y(\tau)S_x(\tau){\bf Z}(t_n)-\tau S_{z}(\tau)S_y(\tau)M_x {\bf Z}(t_n)- \tau S_{z}(\tau)M_y{\bf Z}(t_n)\\[2mm]
&\quad-\tau^2 S_{z}(\tau)S_y(\tau)\alpha_2(\tau M_x)M_x^2{\bf Z}(t_n) -\tau^2 S_{z}(\tau)\alpha_2(\tau M_y)M_y^2{\bf Z}(t_n).
\end{align*}
Substituting it into \eqref{eq 4.2} and using the relation $Id-S_y(\tau)=-\tau M_y \alpha_1(\tau M_y)$ yield
\begin{align*}
&I_1=\Big(S(\tau)-S_{z}(\tau)S_{y}(\tau)S_{x}(\tau)\Big){\bf Z}(t_n)\\[2mm]
&=-\tau^2S_z(\tau) \alpha_1(\tau M_y)M_y M_x {\bf Z}(t_n)
-\tau^2 S_{z}(\tau)S_y(\tau)\alpha_2(\tau M_x)M_x^2{\bf Z}(t_n)\\[2mm]
&\quad-\tau^2 S_{z}(\tau)\alpha_2(\tau M_y)M_y^2{\bf Z}(t_n)
-\int_{t_n}^{t_{n+1}} sS_{z}(s)\big( M_z M_y-M_y M \big) S(t_{n+1}-s){\bf Z}(t_n){\rm d}s\\[2mm]
&\quad-\int_{t_n}^{t_{n+1}} sS_{z}(s)\big( M_z M_x-M_x M \big) S(t_{n+1}-s){\bf Z}(t_n){\rm d}s.
\end{align*}
It follows from the $H^2$-regularity of continuous solution ${\bf Z}(t)$ (see Proposition \ref{pro_2.2}) that 
\begin{equation}\label{eq_4.7}
\mathbb{E}\left\|I_{1}\right\|^2_{V}\leq C\tau^4{\mathbb E}\|{\bf Z}(t_n)\|^2_{H^2(D)^6}\leq C\tau^4.
\end{equation}
\vskip 0.5em
(ii) {\it Estimates of terms $I_2$, $I_3$ and $I_4$}
\vskip 0.5em
For the term $I_2$, It\^o isometry yields
\begin{align}\label{eq_4.8}
&{\mathbb E}\left\|I_2\right\|_V^{2}=  \int_{t_n}^{t_{n+1}}\left\|\Big(S(t_{n+1}-s)-S_{z}(\tau)S_y(\tau)S_x(t_{n+1}-s)\Big)\lambda^{[1]}\circ Q^{\frac12}\right\|_{HS(U,V)}^2{\rm d}s\\[2mm]
&\leq C \int_{t_n}^{t_{n+1}}\Big\| \Big(S_{z}(\tau)S_{y}(\tau)\left( S(t_{n+1}-s)-S_{x}(t_{n+1}-s)\Big)\lambda^{[1]}\circ Q^{\frac12} \right) \Big\|^{2}_{HS(U,V)}{\rm d}s\nonumber\\[2mm]
&\quad\quad+C \int_{t_n}^{t_{n+1}}\Big\|  \Big(S_{z}(\tau)\left(Id-S_y(\tau)\right)S(t_{n+1}-s)\Big)\lambda^{[1]}\circ Q^{\frac12}\Big \|^{2}_{HS(U,V)}{\rm d}s\nonumber\\[2mm]
&\quad\quad+C \int_{t_n}^{t_{n+1}}\Big\| \left(Id-S_{z}(\tau)\right)S(t_{n+1}-s)\lambda^{[1]}\circ Q^{\frac12} \Big\|^{2}_{HS(U,V)}{\rm d}s\nonumber\\[2mm]
&\leq C(|\lambda^{[1]}|, \|Q^{\frac12}\|_{HS(U,H^1(D))})\tau^3,\nonumber
\end{align}
where we have used the fact that $S(t)$ and $S_\alpha(t), \alpha=x,y,z$ are unitary $C_0$-semigroups and the inequalities
\[
\|S(s)-Id\|_{{\mathcal L}({\mathcal D}(M),V)} =s\|\alpha_1(sM)M \|_{{\mathcal L}({\mathcal D}(M),V)} \leq C\tau ,\quad \forall ~s\leq \tau
\]
and
\begin{equation*}\label{err_S}
\|S_{\alpha}(s)-Id\|_{{\mathcal L}({\mathcal D}(M_{\alpha}),V)}=s\|\alpha_1(sM_{\alpha})M_{\alpha} \|_{{\mathcal L}({\mathcal D}(M_{\alpha}),V)}\leq C\tau,\quad \forall ~s\leq \tau.
\end{equation*}
Estimates of  terms $I_3$ and $I_4$ are similar with
 \begin{equation}\label{eq_4.9}
 {\mathbb E}\|I_3\|_V^2+{\mathbb E}\|I_4\|_V^2\leq C(|\lambda^{[2]}|,|\lambda^{[3]}|, \|Q^{\frac12}\|_{HS(U,H^1(D))})\tau^3.
 \end{equation} 

Taking ${\mathbb E}\|\cdot\|_V^2$ on both sides of the error equation \eqref{eq_4.5} and combining estimates \eqref{eq_4.7}-\eqref{eq_4.9}, we obtain
\begin{equation*}
\begin{split}
\mathbb{E}\|e_{n+1}\|_V^2&\leq\mathbb{E}\|S_z(\tau)S_y(\tau)S_x(\tau)e_n\|_{V}^2  + 2{\mathbb E}\langle S_z(\tau)S_y(\tau)S_x(\tau)e_n, I_1 \rangle_{V}\\[2mm]
&\quad+ C\mathbb{E}\left(\|I_1\|_V^2+\|I_2\|_V^2+\|I_3\|_V^2+\|I_4\|_V^2\right)\\[2mm]
&\leq (1+C\tau)\mathbb{E}\|S_z(\tau)S_y(\tau)S_x(\tau)e_n\|_V^2+\frac{C}{\tau}\mathbb{E}\|I_1\|_V^2\\[2mm]
&\quad+C\mathbb{E}\left(\|I_1\|_V^2+\|I_2\|_V^2+\|I_3\|_V^2+\|I_4\|_V^2\right)\\[2mm]
&\leq(1+C\tau)\mathbb{E}\|e_n\|_V^2+C\tau^3,
\end{split}
\end{equation*}
due to the unitarity of semigroups $S_{\alpha}$ ($\alpha=x,y,z$). 

 Gronwall’s inequality completes the proof.
\end{proof}

\subsection{Temporally semi-discretized error}
In this subsection, each subsystem in {\bf Stages 1-3} is discretized temporally  by using 
numerical methods, for example, the implicit Euler method, the midpoint method, the exponential Euler method, etc. 
Notice that the midpoint method and the exponential Euler method preserve the stochastic symplectic structure.
We refer readers to \cite{CHJ2019a, CHJ2019b, CCHS2020} for the analysis of these methods for stochastic Maxwell equations, and to \cite{CHJS2021} for the probabilistic superiority of stochastic symplectic methods.  Let $({\bf E}^{n-1},{\bf H}^{n-1})$ denote the approximated solution at time $t_{n-1}$, and let $({\bf E}^{0},{\bf H}^{0})=({\bf E}_{0},{\bf H}_{0})$ and $\Delta W^n=W(t_n)-W(t_{n-1})$. We take the implicit Euler method for an example, and thus {\bf Stages 1-3}  become:
\begin{itemize}
\item {\bf Stage $1^{\prime}$.} Use the implicit Euler method to temporally solve \eqref{eq 4.55}, 
\begin{align*}
\left\{
\begin{aligned}
E_1^{n-1,*}=E_1^{n-1}+\lambda_1^1 \Delta W^n\\
H_1^{n-1,*}=H_1^{n-1}+\lambda_2^1\Delta W^n
\end{aligned}
\right. ,\quad
\left\{
\begin{aligned}
E_2^{n-1,*}=E_2^{n-1}-\tau \partial_x H_3^{n-1,*}\\
H_3^{n-1,*}=H_3^{n-1}-\tau  \partial_x E_2^{n-1,*}\end{aligned}
\right. ,\\
\left\{
\begin{aligned}
E_3^{n-1,*}=E_3^{n-1}+\tau \partial_x H_2^{n-1,*}\\
H_2^{n-1,*}=H_2^{n-1}+\tau  \partial_x E_3^{n-1,*}
\end{aligned}
\right. .
\end{align*}
\item {\bf Stage $2^{\prime}$.} Use the implicit Euler method to temporally solve \eqref{eq 4.56}, 
\begin{align*}
\left\{
\begin{aligned}
E_2^{n-1,**}=E_2^{n-1,*}+\lambda_1^2 \Delta W^n\\
H_2^{n-1,**}=H_2^{n-1,*}+\lambda_2^2\Delta W^n
\end{aligned}
\right. ,\quad
\left\{
\begin{aligned}
E_1^{n-1,**}=E_1^{n-1,*}+\tau \partial_y H_3^{n-1,**}\\
H_3^{n-1,**}=H_3^{n-1,*}+\tau  \partial_y E_1^{n-1,**}
\end{aligned}
\right. ,\\
\left\{
\begin{aligned}
E_3^{n-1,**}=E_3^{n-1,*}-\tau \partial_y H_1^{n-1,**}\\
H_1^{n-1,**}=H_1^{n-1,*}-\tau  \partial_y E_3^{n-1,**}
\end{aligned}
\right. .
\end{align*}
\item {\bf Stage $3^{\prime}$.} Use the implicit Euler method to temporally solve \eqref{eq 4.57}, 
\begin{align*}
\left\{
\begin{aligned}
E_3^{n}=E_3^{n-1,**}+\lambda_1^3 \Delta W^n\\
H_3^{n}=H_3^{n-1,**}+\lambda_2^3\Delta W^n
\end{aligned}
\right. ,\quad
\left\{
\begin{aligned}
E_1^{n}=E_1^{n-1,**}-\tau \partial_z H_2^{n}\\
H_2^{n}=H_2^{n-1,**}-\tau  \partial_z E_1^{n}\end{aligned}
\right. ,\\
\left\{
\begin{aligned}
E_2^{n}=E_2^{n-1,**}+\tau \partial_z H_1^{n}\\
H_1^{n}=H_1^{n-1,**}+\tau  \partial_z E_2^{n}
\end{aligned}
\right. .
\end{align*}
\end{itemize} 

To avoid the confusion of notations, below we use $\Phi^{IE}_{t,t+\tau}({\bf Z})$, $\Phi^{M}_{t,t+\tau}({\bf Z})$ and $\Phi^{EE}_{t,t+\tau}({\bf Z})$ to denote the temporally semi-discrete solutions at $t+\tau$ starting from ${\bf Z}$ at time $t$ of stochastic Maxwell equations \eqref{sto_max} via the combination of the splitting and the  implicit Euler method, the midpoint method, the exponential Euler method, respectively.

Similar as \eqref{eq_4.3}, the one-step temporal approximations  read:
\begin{itemize}
\item Implicit Euler method
\begin{equation*}
\begin{split}
\Phi^{IE}_{t,t+\tau}({\bf Z})=&S_{\tau,z}^{IE}S_{\tau,y}^{IE}S_{\tau,x}^{IE}{\bf Z}+S_{\tau,z}^{IE}S_{\tau,y}^{IE}S_{\tau,x}^{IE}\lambda^{[1]}\Delta W\\[2mm]
&+S_{\tau,z}^{IE}S_{\tau,y}^{IE}\lambda^{[2]}\Delta W + S_{\tau,z}^{IE}\lambda^{[3]}\Delta W,
\end{split}
\end{equation*}
where {$S_{\tau,\alpha}^{IE}=(Id-\tau M_{\alpha})^{-1}$ with $\alpha=x,y,z$}, and $\Delta W=W(t+\tau)-W(t)$.
\item Midpoint method
\begin{equation*}
\begin{split}
\Phi^{M}_{t,t+\tau}({\bf Z})=&S_{\tau,z}^{M}S_{\tau,y}^{M}S_{\tau,x}^{M}{\bf Z}+S_{\tau,z}^{M}S_{\tau,y}^{M}T_{\tau,x}^{M}\lambda^{[1]}\Delta W\\[2mm]
&+S_{\tau,z}^{M}T_{\tau,y}^{M}\lambda^{[2]}\Delta W + T_{\tau,z}^{M}\lambda^{[3]}\Delta W,
\end{split}
\end{equation*}
where $S_{\tau,\alpha}^{M}=(Id-\frac{\tau}{2} M_{\alpha})^{-1}(Id+\frac{\tau}{2}M_{\alpha})$, $T_{\tau,\alpha}^{M}=(Id-\frac{\tau}{2} M_{\alpha})^{-1}$ with $\alpha=x,y,z$, and $\Delta W=W(t+\tau)-W(t)$.
\item Exponential Euler method
\begin{equation*}
\begin{split}
\Phi^{EE}_{t,t+\tau}({\bf Z})=&S_{z}(\tau)S_{y}(\tau)S_{x}(\tau){\bf Z}+S_{z}(\tau)S_{y}(\tau)S_{x}(\tau)\lambda^{[1]}\Delta W\\[2mm]
&+
S_{z}(\tau)S_{y}(\tau)\lambda^{[2]}\Delta W + S_{z}(\tau)\lambda^{[3]}\Delta W,
\end{split}
\end{equation*}
where $S_{\alpha}(\tau)=e^{\tau M_{\alpha}}$ with $\alpha=x,y,z$, and $\Delta W=W(t+\tau)-W(t)$.
\end{itemize}

For the analysis of the numerical error of the above three numerical methods, the following two lemmas are introduced.
\begin{lemma}\label{lem_001}
It holds
\begin{equation*}
\begin{split}
&\left|S_{\tau,\alpha}^{IE}\right|_{{\mathcal L}(V,V)}\leq 1,\quad
\left|S_{\tau,\alpha}^{M}\right|_{{\mathcal L}(V,V)}= 1,\\
&\left|T_{\tau,\alpha}^{M}\right|_{{\mathcal L}(V,V)}\leq 1,\quad \left|T_{\tau,\alpha}^{M}-Id\right|_{{\mathcal L}(V,V)}\leq C\tau.
\end{split}
\end{equation*}
\end{lemma}
The proof of this lemma is standard, thus is omitted.
\begin{lemma}\label{lem_002}
For the implicit Euler method and the midpoint method, there exists a constant $C$ independent of $\tau$ such that \\
\begin{itemize}
\item [(i)]$
 ~\|S(\tau){\bf Z}-S_{\tau,z}^{IE}S_{\tau,y}^{IE}S_{\tau,x}^{IE}{\bf Z}\|_{V}\leq C\tau^2\|{\bf Z}\|_{H^2(D)^6}.
$\\
\item[(ii)]$
~\|S(\tau){\bf Z}-S_{\tau,z}^{M}S_{\tau,y}^{M}S_{\tau,x}^{M}{\bf Z}\|_{V}\leq C\tau^2\|{\bf Z}\|_{H^2(D)^6}.
$
\end{itemize}
\end{lemma}
\begin{proof} {\it(i)}
It suffices to prove $$\|S_{z}(\tau)S_{y}(\tau)S_{x}(\tau){\bf Z}-S_{\tau,z}^{IE}S_{\tau,y}^{IE}S_{\tau,x}^{IE}{\bf Z}\|_{V}\leq C\tau^2\|{\bf Z}\|_{H^2(D)^6},$$ based on the estimate of the term $I_1$ in the proof of Theorem \ref{thm_4.1}.
For $\alpha=x,y,z$, denote $a=\tau M_{\alpha}$ and $\xi=(Id-a)^{-1}$. Then
\[
\xi=Id+\xi a =Id+(Id+\xi a)a=Id+a+\xi a^2,
\]
i.e., $S_{\tau,\alpha}^{IE}=Id+\tau M_{\alpha}+\tau^2 \xi M_{\alpha}^2$, where $\xi$ is a bounded operator.
Recall the relation
\[
S_{\alpha}(\tau)=Id+\tau M_{\alpha} +\tau^2 \alpha_2(\tau M_{\alpha})M_{\alpha}^2,
\]
we have for any $u\in {\mathcal D}(M_{\alpha}^2)$
\begin{align*}
\|(S_{\alpha}(\tau)-S_{\tau,\alpha}^{IE})u\|_{V}=\|\tau^2 \alpha_2(\tau M_{\alpha})M_{\alpha}^2 u-\tau^2 \xi M_{\alpha}^2 u\|_{V}\leq C\tau^2\|u\|_{{\mathcal D}(M_{\alpha}^2)}.
\end{align*}
Therefore,
\begin{align*}
&\|S_{z}(\tau)S_{y}(\tau)S_{x}(\tau){\bf Z}-S_{\tau,z}^{IE}S_{\tau,y}^{IE}S_{\tau,x}^{IE}{\bf Z}\|_{V}\\
&\leq \| (S_{z}(\tau)-S_{\tau,z}^{IE})S_{y}(\tau)S_{x}(\tau){\bf Z} \|_{V}
+\|S_{\tau,z}^{IE}(S_{y}(\tau)-S_{\tau,y}^{IE})S_{x}(\tau){\bf Z}\|_{V}\\
&\quad+\|S_{\tau,z}^{IE}S_{\tau,y}^{IE}(S_{x}(\tau)-S_{\tau,x}^{IE}){\bf Z}\|_{V}\\
&\leq C\tau^2\|{\bf Z}\|_{H^2(D)^6}.
\end{align*}

{\it (ii)}
Similarly, by denoting  $a=\frac{\tau}{2} M_{\alpha}$ and $\xi=(Id-a)^{-1}$, we have
\begin{align*}
S_{\tau,\alpha}^{M}=&(Id+\frac{\tau}{2}\xi M_{\alpha})(Id+\frac{\tau}{2} M_{\alpha})
=Id+\frac{\tau}{2} M_{\alpha}+\frac{\tau}{2}\xi M_{\alpha}+\frac{\tau^2}{4}\xi M_{\alpha}^2\\
=&Id+\frac{\tau}{2} M_{\alpha}+\frac{\tau}{2}(Id+\frac{\tau}{2}\xi M_{\alpha}) M_{\alpha}+\frac{\tau^2}{4}\xi M_{\alpha}^2
=Id+\tau M_{\alpha}+\frac{\tau^2}{2}\xi M_{\alpha}^2.
\end{align*}
Thus
\begin{align*}
\|(S_{\alpha}(\tau)-S_{\tau,\alpha}^{M})u\|_{V}=\|\tau^2 \alpha_2(\tau M_{\alpha})M_{\alpha}^2 u-\frac{\tau^2}{2} \xi M_{\alpha}^2 u\|_{V}\leq C\tau^2\|u\|_{{\mathcal D}(M_{\alpha}^2)}.
\end{align*}
The rest step is similar as in {\it (i)}. The proof is completed.
\end{proof}

Denote the  discrete composed solution at $t_n$ by $${{\bf Z}}^n=\Phi_{t_{n-1},t_n}\circ \Phi_{t_{n-2},t_{n-1}}\circ\cdots\circ \Phi_{t_{0},t_1}({\bf Z}_0)$$  with $\Phi_{t,t+\tau}\in\{\Phi^{IE}_{t,t+\tau},~\Phi^{M}_{t,t+\tau},~\Phi^{EE}_{t,t+\tau}\}$. Set ${\hat e}^n:={\bf Z}(t_n)-{{\bf Z}}^n$. The mean square convergence result is stated below.
\begin{theorem}
Under   the same assumption as in Proposition \ref{pro_2.2}, for sufficiently small $\tau$, there exists a constant $C$ independent of $\tau$ such that
\[
\max_{0\leq n\leq N}\Big( {\mathbb E} \left\| \hat{e}^n\right\|_{V}^{2} \Big)^{1/2}\leq C\tau.
\]
\end{theorem}
\begin{proof}
We take the midpoint method as an example, since the proofs of the implicit Euler method and the exponential Euler method are similar.
Recall that 
\begin{equation*}
\begin{split}
{{\bf Z}}^{n+1}=&S_{\tau,z}^{M}S_{\tau,y}^{M}S_{\tau,x}^{M}{{\bf Z}}^n+S_{\tau,z}^{M}S_{\tau,y}^{M}T_{\tau,x}^{M}\lambda^{[1]}\Delta W^{n+1}\\[2mm]
&+S_{\tau,z}^{M}T_{\tau,y}^{M}\lambda^{[2]}\Delta W^{n+1} + T_{\tau,z}^{M}\lambda^{[3]}\Delta W^{n+1},
\end{split}
\end{equation*}
where $\Delta W^{n+1}=W(t_{n+1})-W(t_n)$.
Subtracting it from \eqref{eq_4.4} leads to
\begin{equation*}
{\hat e}^{n+1}=S_{\tau,z}^{M}S_{\tau,y}^{M}S_{\tau,x}^{M}\hat{e}^n
+{\hat I}_1+{\hat I}_2+{\hat I}_3+{\hat I}_4,
\end{equation*}
with  
\begin{equation*}
\begin{split}
{\hat I}_1=&\Big(S(\tau)-S_{\tau,z}^{M}S_{\tau,y}^{M}S_{\tau,x}^{M}\Big){\bf Z}(t_n),\\[2mm]
{\hat I}_2=&\int_{t_n}^{t_{n+1}}\Big[S(t_{n+1}-r)-S_{\tau,z}^{M}S_{\tau,y}^{M}T_{\tau,x}^{M}\Big]\lambda^{[1]}{\rm d}W(r),\\[2mm]
{\hat I}_3=&\int_{t_n}^{t_{n+1}}\Big[S(t_{n+1}-r)-S_{\tau,z}^{M}T_{\tau,y}^{M}\Big]\lambda^{[2]}{\rm d}W(r),\\[2mm]
{\hat I}_4=&\int_{t_n}^{t_{n+1}}\Big[S(t_{n+1}-r)-T_{\tau,z}^{M}\Big]\lambda^{[3]}{\rm d}W(r).
\end{split}
\end{equation*}
Utilizing Lemma \ref{lem_002} {\it (ii)},
\begin{equation*}
\begin{split}
\mathbb{E}\|{\hat I}_1\|_V^2
\leq C\tau^4\mathbb{E}\|{\bf Z}(t_n)\|_{H^2(D)^6}^2\leq C\tau^4.
\end{split}
\end{equation*}
For the term ${\hat I}_2$,
\begin{align*}
{\mathbb E}\left\|{\hat I}_2\right\|_V^{2}&=  \int_{t_n}^{t_{n+1}}\left\|\Big(S(t_{n+1}-s)-S_{\tau,z}^{M}S_{\tau,y}^{M}T_{\tau,x}^{M}\Big)\lambda^{[1]}\circ Q^{\frac12}\right\|_{HS(U,V)}^2{\rm d}s\\[2mm]
&\leq C \int_{t_n}^{t_{n+1}}\Big\| \Big(S_{\tau,z}^{M}S_{\tau,y}^{M}\left( S(t_{n+1}-s)-T_{\tau,x}^{M}\Big)\lambda^{[1]}\circ Q^{\frac12} \right) \Big\|^{2}_{HS(U,V)}{\rm d}s\\[2mm]
&\quad\quad+C \int_{t_n}^{t_{n+1}}\Big\|  \Big(S_{\tau,z}^{M}\left(Id-S_{\tau,y}^{M}\right)S(t_{n+1}-s)\Big)\lambda^{[1]}\circ Q^{\frac12}\Big \|^{2}_{HS(U,V)}{\rm d}s\\[2mm]
&\quad\quad+C \int_{t_n}^{t_{n+1}}\Big\| \left(Id-S_{\tau,z}^{M}\right)S(t_{n+1}-s)\lambda^{[1]}\circ Q^{\frac12} \Big\|^{2}_{HS(U,V)}{\rm d}s\\[2mm]
&\leq C(|\lambda^{[1]}|, \|Q^{\frac12}\|_{HS(U,H^1(D))})\tau^3,
\end{align*}
due to Lemma \ref{lem_001}.  Similarly, we can obtain
\begin{equation*}
{\mathbb E}\left\|{\hat I}_3\right\|_V^{2}\leq C(|\lambda^{[2]}|,\|Q^{\frac12}\|_{HS(U,H^1(D))})\tau^3,~~ {\mathbb E}\left\|{\hat I}_4\right\|_V^{2}\leq C(|\lambda^{[3]}|, \|Q^{\frac12}\|_{HS(U,H^1(D))})\tau^3.
\end{equation*}
Combing the above results, we get
\begin{equation*}
\begin{split}
\mathbb{E}\|\hat{e}_{n+1}\|_V^2&\leq\mathbb{E}\|S_{\tau,z}^{M}S_{\tau,y}^{M}S_{\tau,x}^{M}\hat{e}_n\|_{V}^2  + 2{\mathbb E}\langle S_{\tau,z}^{M}S_{\tau,y}^{M}S_{\tau,x}^{M}\hat{e}_n, \hat{I}_1 \rangle_{V}\\[2mm]
&\quad+ C\mathbb{E}\left(\|\hat{I}_1\|_V^2+\|\hat{I}_2\|_V^2+\|\hat{I}_3\|_V^2+\|\hat{I}_4\|_V^2\right)\\[2mm]
&\leq (1+C\tau)\mathbb{E}\|S_{\tau,z}^{M}S_{\tau,y}^{M}S_{\tau,x}^{M}\hat{e}_n\|_V^2+\frac{C}{\tau}\mathbb{E}\|\hat{I}_1\|_V^2\\[2mm]
&\quad+C\mathbb{E}\left(\|\hat{I}_1\|_V^2+\|\hat{I}_2\|_V^2+\|\hat{I}_3\|_V^2+\|\hat{I}_4\|_V^2\right)\\[2mm]
&\leq(1+C\tau)\mathbb{E}\|e_n\|_V^2+C\tau^3.
\end{split}
\end{equation*}
Thus the proof is completed by using Gronwall's inequality.
\end{proof}

\bibliographystyle{plain}
\bibliography{maxwell}
\end{document}